\documentclass[11pt]{amsart}
\usepackage{amsmath}
\usepackage{amssymb}
\usepackage{latexsym}
\usepackage{graphicx}

\newcommand{\N}{\mathbb{N}}

\newtheorem{theorem}{Theorem}
\newtheorem{corollary}{Corollary}
\newtheorem{lemma}{Lemma}
\theoremstyle{remark}
\newtheorem{remark}{Remark}
\newtheorem{example}{Example}

\pagespan{9} {18}

\begin{document}

\title[STARLIKE MAPPINGS]{\large  ON SOME STARLIKE MAPPINGS INVOLVING CERTAIN CONVOLUTION OPERATORS}

\author[K. O. BABALOLA]{K. O. BABALOLA}

\begin{abstract}
This paper makes a modest contribution to the family of starlike univalent mappings in the open unit disk, by the introduction of some new subclasses of them via certain convolution operators. A new univalence condition is given with examples. Some basic characterizations of functions of the new subclasses are also mentioned.
\end{abstract}



\maketitle

\section{Backgound}
First geometry: A domain $D\subset\mathbb{C}$ is said to be starlike with respect to a point $\xi_0\in D$ if the linear segment joining $\xi_0$ to every other point $\xi\in D$ lies entirely in $D$. In more picturesque language, every point of $D$ is visible from $\xi_0$. If $\xi_0=0$, the origin, $D$ is simply called starlike \cite{PL}.

Let $A$ denote the class of functions:
\[
f(z)=z+a_2z^2+...
\]
which are analytic in the open unit disk $E=\{z\in \mathbb{C}: |z|<1\}$. A function $f\in A$ is called starlike if and only if $f$ maps $E$ onto  a starlike domain. Analytically, $f$ is starlike if and only if, for $z\in E$,
\begin{equation}
Re\;\frac{zf'(z)}{f(z)}>0.\, \label{1}
\end{equation}
(see \cite{MI}).

Let $S^\ast$ denote the family of functions starlike in $E$. This family of univalent mappings has recieved so much attention in geometric functions theory, having found applications in many physical problems. In the sequel we employ convolution operators defined in \cite{KO} to identify new subclasses of this important family. A principal discovery in this paper is a new univalence condition in the open unit disk.

Let $f\in A$. For $g(z)=z+b_2z^2+... \in A$, the convolution (or Hadamard
product) of $f(z)$ and $g(z)$ (written as $(f*g)(z)$) is defined as
\[
(f*g)(z)=z+\sum_{k=2}^\infty a_kb_kz^k.
\]
In \cite{KO} we defined the operators $L_n^\sigma:A\rightarrow A$ (using the convolution
$*$) as follows:
\[
L_n^\sigma f(z)=(\tau_\sigma*\tau_{\sigma,n}^{(-1)}*f)(z).
\]
where
\[
\tau_{\sigma,n}(z)=\frac{z}{(1-z)^{\sigma
-(n-1)}},\;\;\sigma-(n-1)>0,
\]
$\tau_\sigma=\tau_{\sigma,0}$ and $\tau_{\sigma,n}^{(-1)}$ such that
\[
(\tau_{\sigma,n}*\tau_{\sigma,n}^{(-1)})(z)=\frac{z}{1-z}
\]
for a fixed real number $\sigma$ and $n\in \N$. We noted that $L_0^\sigma f(z)=L_0^0 f(z)=f(z)$, $L_1^1
f(z)=zf^{\prime}(z)$. Given $f\in A$, we have
\begin{align}
L_n^\sigma f(z)
&=z+\sum_{k=2}^\infty\left\{\prod_{j=0}^{n-1}\left(\frac{\sigma+k-1-j}{\sigma-j}\right)\right\}a_kz^k\nonumber \\
&=z+\sum_{k=2}^\infty\left\{\frac{(\sigma+k-1)!}{\sigma!}\frac{(\sigma-n)!}{(\sigma+k-1-n)!}\right\}a_kz^k\, \label{2}
\end{align}
(see \cite{KO}).

Similarly we defined $l_n^\sigma:A\rightarrow A$ as follows:
\[
l_n^\sigma f(z)=(\tau_{\sigma}^{(-1)}*\tau_{\sigma,n}*f)(z)
\]
so that for $f\in A$ we have
\[
L_n^\sigma (l_n^\sigma f(z))=l_n^\sigma (L_n^\sigma f(z))=f(z).
\]

In this paper we further restrict $\sigma$ and $n$ by requiring $\sigma\geq n+1$ in the definitions above. Thus the following relations hold:
\begin{equation}
(\sigma-n)L_{n+1}^\sigma f(z)=(\sigma-(n+1))L_n^\sigma f(z)+z(L_n^\sigma f(z))'\, \label{3}
\end{equation}
from which we have
\begin{equation}
(\sigma-n)(L_{n+1}^\sigma f(z))'=(\sigma-n)(L_n^\sigma f(z))'+z(L_n^\sigma f(z))''.\, \label{4}
\end{equation}

Now we say:
\medskip

 {\sc Definition.} A function $f\in A$ belongs to the class $S_n^\sigma$ if and only if
\[
Re\frac{L_{n+1}^\sigma f(z)}{L_n^\sigma f(z)}>\frac{\sigma-(n+1)}{\sigma-n},\;\;\;\sigma\geq n+1,\;\;n\in\N.
\]

Observing from ~(\ref{3}) that
\[
\frac{L_{n+1}^\sigma f(z)}{L_n^\sigma f(z)}=\frac{\sigma-(n+1)}{\sigma-n}+\frac{z(L_n^\sigma f(z))'}{(\sigma-n)L_n^\sigma f(z)}
\]
and the fact that $L_0^\sigma f(z)=f(z)$ we have the following remarks.

\begin{remark}

{\rm(a)} $S_0^\sigma\equiv S^\ast$ and

{\rm(b)} $f\in S_n^\sigma$ if and only if $L_n^\sigma f(z)$ is starlike.

\end{remark}

Following from inclusion relations we will deduce that all members of the class $S_n^\sigma$ are starlike univalent in $E$. We also give a univalence condition based on the inclusion relations and provide some examples. Furthermore we give some basic properties of the class, namely, integral representation, coefficient inequalities, closure under certain integral operators. 

This paper is organised as follows: In the next section we give some preliminary lemmas while Section 3 contains the main results characterizing $S_n^\sigma$. 

\section{Preliminary Lemmas}

Let $P$ denote the class of functions $p(z)=1+c_1z+c_2z^2+...$ which are regular in $E$ and satisfy Re $p(z)>0$, $z\in E$.

\begin{lemma}[\cite{BO,SM}]
Let $u=u_1+u_2i$, $v=v_1+v_2i$ and $\psi$ a complex-valued function satisfying: {\rm(a)} $\psi(u,v)$ is continuous in a domain $\Omega$ of $\mathbb{C}^2$, {\rm(b)} $(1,0)\in\Omega$ and {\rm Re} $\psi(1,0)>0$ and {\rm(c)} {\rm Re} $\psi(u_2i, v_1)\leq 0$ when $(u_2i, v_1)\in\Omega$ and $2v_1\leq -(1+u_2^2)$. If $p(z)=1+c_1z+c_2z^2+...$ satisfies $(p(z),zp'(z))\in\Omega$ and {\rm Re} $\psi(p(z),zp'(z))>0$ for $z\in E$, then {\rm Re} $p(z)>0$ in $E$.
\end{lemma}

\begin{lemma}[\cite{EE}]
Let $\eta$ and $\mu$ be complex constants and $h(z)$ a convex univalent function in $E$ satisfying $h(0)=1$, and {\rm Re} $(\eta h(z)+\mu)>0$. Suppose $p\in P$ satisfies the differential subordination:
\begin{equation}
p(z)+\frac{zp'(z)}{\eta p(z)+\mu}\prec h(z),\;\;\;z\in E.\, \label{5}
\end{equation}
If the differential equation:
\begin{equation}
q(z)+\frac{zq'(z)}{\eta q(z)+\mu}=h(z),\;\;\;q(0)=1\, \label{6}
\end{equation}
has univalent solution $q(z)$ in $E$, then $p(z)\prec q(z)\prec h(z)$ and $q(z)$ is the best dominant in $~(\ref{5})$.
\end{lemma}

The formal solution of ~(\ref{6}) is given as
\begin{equation}
q(z)=\frac{zF'(z)}{F(z)}\, \label{7}
\end{equation}
where
\[
F(z)^\eta=\frac{\eta+\mu}{z^\mu}\int_0^zt^{\mu-1}H(t)^\eta dt
\]
and
\[
H(z)=z.\exp\left(\int_0^z\frac{h(t)-1}{t}dt\right)
\]
(see \cite{SS, HM}). The conditions for the univalence of the solution $q(z)$ of ~(\ref{6}) (given by ~(\ref{7})) as given in \cite{SS} are that

\begin{lemma}[\cite{SS}]
Let $\eta\neq 0$ and $\mu$ be complex constants, and $h(z)$ regular in $E$ with $h'(0)\neq 0$, then the solution $q(z)$ of $~(\ref{6})$ (given by $~(\ref{7})$) is univalent in $E$ if {\rm(i)} {\rm Re} $\{G(z)=\eta h(z)+\mu\}>0$ and {\rm(ii)} $Q(z)=zG'(z)/G(z)$ and $R(z)=Q(z)/G(z)$ are both starlike in $E$.
\end{lemma}

\section{Main results}

\begin{theorem}
Let $\sigma\geq n+1$ and $h(z)$ a convex univalent function in $E$ satisfying $h(0)=1$, and {\rm Re} $(\sigma-(n+1)+ h(z))>0$, $z\in E$. Let $f\in A$. If $\frac{L_{n+2}^\sigma f(z)}{L_{n+1}^\sigma f(z)}\prec h(z)$, then $\frac{L_{n+1}^\sigma f(z)}{L_n^\sigma f(z)}\prec h(z)$.
\end{theorem}

\begin{proof}
By Remark 1(b), it is sufficient to prove that: if 
\[
\frac{z(L_{n+1}^\sigma f(z))'}{L_{n+1}^\sigma f(z)}\prec h(z),
\]
then
\[
\frac{z(L_n^\sigma f(z))'}{L_n^\sigma f(z)}\prec h(z).
\]
Now let
\[
p=\frac{z(L_n^\sigma f(z))'}{L_n^\sigma f(z)}.
\]
Then
\[
z(L_n^\sigma f(z))''+(L_n^\sigma f(z))'=p'(z)L_n^\sigma f(z)+p(z)(L_n^\sigma f(z))'
\]
Using ~(\ref{4}) we obtain
\[
(\sigma-n)z(L_{n+1}^\sigma f(z))'=zp'(z)L_n^\sigma f(z)+zp(z)(L_n^\sigma f(z))'+(\sigma-(n+1))z(L_n^\sigma f(z))'
\]
so with ~(\ref{3}) we have
\begin{align}
\frac{z(L_{n+1}^\sigma f(z))'}{L_{n+1}^\sigma f(z)}
&=\frac{zp'(z)L_n^\sigma f(z)+zp(z)(L_n^\sigma f(z))'+(\sigma-(n+1))z(L_n^\sigma f(z))'}{(\sigma-(n+1))L_n^\sigma f(z)+z(L_n^\sigma f(z))'}\nonumber \\
&=\frac{zp'(z)+p(z)^2+(\sigma-(n+1))p(z)}{\sigma-(n+1)+p(z)}\nonumber \\
&=p(z)+\frac{zp'(z)}{\sigma-(n+1)+p(z)}.\, \label{8}
\end{align}
Now take $\eta=1$ and $\mu=\sigma-(n+1)$ in Lemma 2, the result follows.
\end{proof}

\begin{theorem}
Let $\sigma\geq n+1$ and $h(z)$ a convex univalent function in $E$ satisfying $h(0)=1$, and {\rm Re} $(\sigma-(n+1)+ h(z))>0$, $z\in E$. Let $f\in A$. If $f\in S_{n+1}^\sigma$, then 
\[
\frac{L_{n+1}^\sigma f(z)}{L_n^\sigma f(z)}\prec q(z)
\]
where
\begin{equation}
q(z)=\frac{1+\sum_{k=1}^\infty\frac{\sigma-n}{\sigma-n+k}(k+1)^2z^k}{1+\sum_{k=1}^\infty\frac{\sigma-n}{\sigma-n+k}(k+1)z^k}\, \label{9}
\end{equation}
and it is the best dominant.
\end{theorem}

\begin{proof}
By Remark 1(b) also, if $f\in S_{n+1}^\sigma$, then
\[
\frac{z(L_{n+1}^\sigma f(z))'}{L_{n+1}^\sigma f(z)}\prec\frac{1+z}{1-z}.
\]
So we have to prove that
\[
\frac{z(L_n^\sigma f(z))'}{L_n^\sigma f(z)}\prec q(z).
\]
By considering the differential equation
\[
q(z)+\frac{zq'(z)}{\sigma-(n+1)+q(z)}=\frac{1+z}{1-z}
\]
whose solution (using ~(\ref{7})) is given by ~(\ref{9}), our result follows from Lemma 2 if we prove that $q(z)$ (given by ~(\ref{9})) is univalent in $E$. Now set $\eta=1$, $\mu=\sigma-(n+1)$ and $h(z)=(1+z)/(1-z)$ in Lemma 3, we have

{\rm (i)}\[Re\;G(z)=Re\;[\mu+h(z)]>\mu\geq 0.\]

{\rm (ii)}\[Q(z)=\frac{zG'(z)}{G(z)}=\frac{2}{1+\mu}\frac{z}{(1+az)(1-z)}\]
where $a=(1-\mu)/(1+\mu)$, so that
\begin{align*}
\frac{zQ'(z)}{Q(z)}
&=\frac{1+az^2}{(1+az)(1-z)}\\
&=\frac{1}{1-z}+\frac{1}{1+az}-1.
\end{align*}
Thus Re $zQ'(z)/Q(z)>\mu/2>0$. And finally we have
\[
R(z)=\frac{Q(z)}{G(z)}=\frac{2}{1+\mu}\frac{z}{(1+az)^2}
\]
so that $zR'(z)/R(z)=(1-az)/(1+az)$ with real part greater than zero. Thus $q(z)$ satisfies all conditions of Lemma 3, hence univalent in $E$. This completes the proof.
\end{proof}

\begin{theorem}
\[
S_{n+1}^\sigma\subset S_n^\sigma,\;\;\;n\in\N.
\]
\end{theorem}
\begin{proof}
Let $f\in S_{n+1}^\sigma$. From ~(\ref{8}) above, define $\psi(p(z),zp'(z))=p(z)+\tfrac{zp'(z)}{\sigma-(n+1)+p(z)}$ for $\Omega=[\mathbb{C}-\{-(\sigma-(n+1))\}]\times\mathbb{C}$. Obviously $\psi$ satisfies the conditions (a) and (b) of Lemma 1. Now $\psi(u_2i,v_1)=u_2i+\tfrac{v_1}{\sigma-(n+1)+u_2i}$ so that Re $\psi(u_2i,v_1)=\tfrac{(\sigma-(n+1))v_1}{(\sigma-(n+1))^2+u_2^2}\leq 0$ if $v_1\leq-\tfrac{1}{2}(1+u_2^2)$. Hence by Lemma 1, we have Re $\frac{z(L_{n+1}^\sigma f(z))'}{L_{n+1}^\sigma f(z)}>0$ implies Re $\frac{z(L_n^\sigma f(z))'}{L_n^\sigma f(z)}>0$. By Remark 1(b)
\[
Re\;\frac{L_{n+1}^\sigma f(z)}{L_n^\sigma f(z)}>\frac{\sigma-(n+1)}{\sigma-n}.
\]
This completes the proof.
\end{proof}

\begin{corollary}
All functions in $S_n^\sigma$ are starlike univalent in $E$.
\end{corollary}

Following from the inclusion relations, setting $\sigma=2$ and $n=1$, we have the following important univalence condition.

\begin{corollary}
Let $f\in A$ satisfy
\[
Re\;\frac{2zf'(z)+z^2f''(z)}{f(z)+zf'(z)}>0,\;\;\;z\in E.
\]
Then f(z) is starlike univalent in $E$.
\end{corollary}

\begin{example}
The functions $f_j(z)$ $j=1$, 2, 3, 4., given by
\[
f_1(z)=\frac{2[1-(1-z)e^z]}{z},\;\;\;f_2(z)=\frac{2[1-(1+z)e^{-z}]}{z},
\]
\[
f_3(z)=\frac{-2[z+\log(1-z)]}{z},\;\;\;f_4(z)=\frac{2[z-\log(1+z)]}{z}.
\]
are starlike univalent in the open unit disk.
\end{example}

\begin{proof}
By direct computation we find that
\[
\frac{2zf_j'(z)+z^2f_j''(z)}{f_j(z)+zf_j'(z)}= \left\{
\begin{array}{ll}
1+z&\mbox{if $j=1$},\\
1-z&\mbox{if $j=2$},\\
\frac{1}{1-z}&\mbox{if $j=3$},\\
\frac{1}{1+z}&\mbox{if $j=4$.}
\end{array}
\right.
\]

Observe that the right hand side of the above equations are all functions in $P$, hence we have
\[
Re\;\frac{2zf_j'(z)+z^2f_j''(z)}{f_j(z)+zf_j'(z)}>0,\;\;\;j=1,\;2,\;3,\;4
\]
so that $f_j\in S_1^2\subset S^\ast$. The proof is complete.
\end{proof}

\begin{theorem}
Functions in $S_n^\sigma$ have integral representation:
\[
f(z)=l_n^\sigma\left\{z.\exp\left(\int_0^z\frac{p(t)-1}{t}dt\right)\right\}
\]
for some $p\in P$.
\end{theorem}

\begin{proof}
Let $f\in S_n^\sigma$, then for some $p\in P$ we have
\begin{align*}
\frac{L_{n+1}^\sigma f(z)}{L_n^\sigma f(z)}
&=\frac{\sigma-(n+1)}{\sigma-n}+\frac{z(L_n^\sigma f(z))'}{(\sigma-n)L_n^\sigma f(z)}\\
&=\frac{\sigma-(n+1)}{\sigma-n}+\frac{p(z)}{(\sigma-n)}.
\end{align*}
Hence we have
\[
\frac{z(L_n^\sigma f(z))'}{L_n^\sigma f(z)}=p(z).
\]
Simple calculation now leads to
\[
L_n^\sigma f(z)=z.\exp\left(\int_0^z\frac{p(t)-1}{t}dt\right).
\]
Applying $l_n^\sigma$ on both sides we have the representation.
\end{proof}

If we choose $p(z)=(1+z)/(1-z)$, we obtain the leading example of the class $S_n^\sigma$, which is
\begin{equation}
k_n^\sigma(z)=z+\sum_{k=2}^\infty\left\{\frac{\sigma!}{(\sigma+k-1)!}\frac{(\sigma+k-1-n)!}{(\sigma-n)!}\right\}kz^k.\, \label{10}
\end{equation}

Next we investigate the closure property of the class $S_n^\sigma$ under the Bernardi integral transformation:
\begin{equation}
F(z)=\frac{\gamma+1}{z^\gamma}\int_0^zt^{\gamma-1}f(t)dt,\;\;\;\gamma>-1.\, \label{11}
\end{equation}

The well known Libera integral corresponds to $\gamma = 1$.

\begin{theorem}
The class $S_n^\sigma$ is closed under $F$.
\end{theorem}
\begin{proof}
From ~(\ref{11}) we have
\begin{equation}
\gamma F(z)+zF'(z)=(\gamma+1)f(z).\, \label{12}
\end{equation}
If we apply $L_n^\sigma$ on ~(\ref{12}), noting from ~(\ref{2}) that $L_n^\sigma(zF'(z))=z(L_n^\sigma F(z))'$ we have
\[
\frac{z(L_n^\sigma f(z))'}{L_n^\sigma f(z)}=\frac{(\gamma+1)z(L_n^\sigma F(z))'+z^2(L_n^\sigma F(z))''}{\gamma L_n^\sigma F(z)+z(L_n^\sigma F(z))'}.
\]
Let $p(z)=z(L_n^\sigma f(z))'/L_n^\sigma f(z)$. Then by some calculation we find that
\[
\frac{z(L_n^\sigma f(z))'}{L_n^\sigma f(z)}=p(z)+\frac{zp'(z)}{\gamma+p(z)}=\psi(p(z),zp'(z))
\]
Define $\psi(p(z),zp'(z))=p(z)+\tfrac{zp'(z)}{\gamma+p(z)}$ for $\Omega=[\mathbb{C}-\{-\gamma\}]\times\mathbb{C}$. Then, as in Theorem 2, $\psi$ satisfies all the conditions of Lemma 1, hence $\frac{z(L_n^\sigma f(z))'}{L_n^\sigma f(z)}>0$ implies Re $\frac{z(L_n^\sigma F(z))'}{L_n^\sigma F(z)}>0$ and by Remark 1(b) we have
\[
Re\;\frac{L_{n+1}^\sigma F(z)}{L_n^\sigma F(z)}>\frac{\sigma-(n+1)}{\sigma-n}
\]
as required.
\end{proof}
\begin{theorem}
Let $f\in S_n^\sigma$. Then we have the inequalities
\[
|a_k|\leq\frac{\sigma!}{(\sigma+k-1)!}\frac{(\sigma+k-1-n)!}{(\sigma-n)!}k,\;\;\;k\geq 2.
\]
The function $k_n^\sigma(z)$, given by $~(\ref{10})$, show that the inequalities are sharp.
\end{theorem}

\begin{proof}
It is known that for each $f\in S^\ast$, $|a_k|\leq k$, $k\geq 2$. Thus by Remark 1(b), for each $f\in S_n^\sigma$ the coefficients of $L_n^\sigma f(z)$ satisfy $|a_k|\leq k$, $k\geq 2$. Hence using ~(\ref{2}) we have the inequalities.
\end{proof}

{\it Acknowledgements.} This work was carried out at the Centre for Advanced Studies in Mathematics, CASM, Lahore University of Management Sciences, Lahore, Pakistan during the author's postdoctoral fellowship at the Centre. The author is indebted to all staff of CASM for their hospitality, most especially Prof. Ismat Beg.

\vspace{10pt}

\hspace{-4mm}{\small{Received}}

\vspace{-12pt}
\ \hfill \
\begin{tabular}{c}
{\small\em  {\bf Current Address}}\\
{\small\em  Centre for Advanced Studies in Mathematics}\\
{\small\em  Lahore University of Management Sciences}\\
{\small\em  Lahore, Pakistan}\\
{\small\em E-mail: {\tt kobabalola@lums.edu.pk}}\\
{\small\em  {\bf Permanent Address}}\\
{\small\em  Department of Mathematics}\\
{\small\em  University of Ilorin}\\
{\small\em  Ilorin, Nigeria}\\
{\small\em E-mail: {\tt kobabalola@gmail.com}}\\
\end{tabular}


\begin{thebibliography}{9}

\bibitem {BO}
\textsc{Babalola, K. O.} and \textsc{Opoola, T. O.}, {\it Iterated
integral transforms of Caratheodory functions and their
applications to analytic and univalent functions}, Tamkang J.
Math., {\bf37} (4) (2006), 355--366.

\bibitem {KO}
\textsc{Babalola, K. O.} {\it New subclasses of analytic and univalent functions involving certain convolution opeartor}, Mathematica, Tome 50 (73) No. 1, (2008), 3--12.

\bibitem {PL}
\textsc{Duren, P. L.}, {\it Univalent functions}, Springer Verlag.
New York Inc. 1983.

\bibitem {EE}
\textsc{Eenigenburg, P., Miller, S. S., Mocanu, P. T.} and \textsc{Reade, M. O.}, {\it On a Briot-Bouquet differential surbordination}, Rev. Roumaine Math. Pures Appl., {\bf29} (1984), 567--573.

\bibitem {SM}
\textsc{Miller, S. S.} and \textsc{Mocanu, P. T.}, {\it Second-order differential inequalities in the complex plane}, J. Math. Anal. Appl. {\bf65} (1978), 289--305.

\bibitem {SS}
\textsc{Miller, S. S.} and \textsc{Mocanu, P. T.}, {\it Univalent solution of Briot-Bouquet differential equations}, Lecture Notes in Mathematics, Springer Berlin/Heidelberg {\bf1013} (1983), 292--310.

\bibitem {MI}
\textsc{Robertson, M. I. S.}, {\it On the theory of univalent functions}, Ann. of Math., {\bf37} (1936), 374--408.

\bibitem {SR}
\textsc{Ruscheweyh, S.}, {\it Convolutions in geometric function theory}, Les Presses de l'Universite de Montreal, 1982.

\bibitem {HM}
\textsc{Srivastava, H. M.} and \textsc{Lashin, A. Y.}, {\it Some applications of the Briot-Bouquet differential subordination}, J. Inequal. Pure and Appl. Math., {\bf6}(2) Art. 41 (2005), 1--7.

\end{thebibliography}
\end{document}